\documentclass[12pt,draft]{amsart}
\usepackage{graphicx,color}

\usepackage{amssymb,amsmath}   
\usepackage{amsthm}						 
\usepackage{color}             
\usepackage{a4wide}            
\usepackage[all]{xy}
\usepackage{mathabx}
\usepackage{stmaryrd}

\DeclareMathOperator*{\bigdotplus}{\scalerel*{\dotplus}{\sum}}

\usepackage{scalerel}

\newcommand{\F}{\ensuremath{\mathbb{F}}}
\newcommand{\K}{\ensuremath{\mathbb{K}}}
\newcommand{\der}[1]{{\sf Der}(#1)}
\newcommand{\gl}[1]{{\mathfrak g\mathfrak l}(#1)}

\newcommand{\id}{{\rm id}}

\theoremstyle{plain}

\newtheorem{teo}{Theorem}[section]
\newtheorem*{*teo}{Theorem}

\newtheorem{lem}[teo]{Lemma}

\newcommand{\Char}[1]{{\rm char}\,#1}

\newtheorem{prop}[teo]{Proposition}

\theoremstyle{definition}
\newtheorem{df}[teo]{Definition}
\newtheorem{ex}[teo]{Example}

\newcommand{\ad}[2]{{\sf ad}^{#1}_{#2}}

\newcommand{\comp}[2]{{\sf Comp}(#1,#2)}
\newcommand{\sdsum}{\oright}
\newcommand{\ennd}[1]{{\sf End}(#1)}

\renewcommand{\bar}[1]{\overline{#1}}
\newcommand{\xp}{(x,p)}

\title[Non-singular derivations in prime characteristic]{Non-singular derivations of solvable Lie algebras
  in prime characteristic}
\author{Marcos Goulart Lima}
\author{Csaba Schneider}

\address[Schneider]{Departamento de Matem\'atica\\
Instituto de Ci\^encias Exatas\\
Universidade Federal de Minas Gerais\\
Av.\ Ant\^onio Carlos 6627\\
Belo Horizonte, MG, Brazil. Email:
{\rm\texttt{csaba@mat.ufmg.br}}, URL:
{\tt
  {http://www.mat.ufmg.br/$\sim$csaba}}}

\address[Lima]
{Departamento de Ci\^encias Exatas e Aplicadas\\
  Universidade Federal de Ouro Preto (Campus Jo\~ao Monlevade)\\
  Rua 36, $\mbox{n}^{{\rm o}}$ 115, Loanda, Jo\~ao Monlevade, MG, CEP 35931-008, Brazil
  Email: {\rm\texttt{marcosgoulart@ufop.edu.br}}}



\thanks{Most of the work
  presented in this paper was carried out as the first author's
  PhD project; he thanks the {\em Universidade Federal de Ouro Preto}\/
  for financially supporting his studies. The second
  author was a recepient of the CNPq grants 308773/2016-0
  ({\em Bolsa de Produtividade em Pesquisa}\/)
  and 421624/2018-3 ({\em Universal}\/). We thank the referee for his or her
  useful suggestions, particularly for those that simplified the arguments in Section~\ref{sec:examples} and in the proof of Lemma~\ref{4.1.9}}

\begin{document}
\begin{abstract}
  We study solvable Lie algebras in prime characteristic $p$
  that admit non-singular
derivations. We show that Jacobson's Theorem remains
true if the quotients of the derived series have dimension less than~$p$.
We also study the structure of
Lie algebras with non-singular derivations in which the derived subalgebra
is abelian and has codimension~one. The paper presents some new examples of
solvable, but not nilpotent, Lie algebras of derived length~3 with non-singular
derivations.
\end{abstract}
\maketitle





\section{Introduction}

By Jacobson's famous theorem~\cite[Theorem~3]{Jacobson},
a finite-dimensional Lie algebra
over a field of characteristic~zero with a non-singular derivation
is nilpotent. More recently~\cite{KK,BM} showed that a finite-dimensional
Lie algebra in characteristic zero that admits a periodic derivation
(that is, a derivation of finite multiplicative order) has nilpotency class at
most two; moreover, if the order of the derivation is not a
multiple of 6, then the Lie algebra is abelian.
As shown by examples in~\cite{Shalev,Mattarei},
Jacobson's Theorem fails in positive characteristic.
A classification of finite-dimensional simple
Lie algebras with non-singular derivations was given in~\cite{Benkart}.
In this paper, we explore the structure of finite-dimensional solvable, but
not nilpotent, Lie algebras that admit a non-singular derivation.

We present some auxiliary results in Section~\ref{chap2} before
exhibiting, in Section~\ref{sec:examples}, new examples of non-nilpotent Lie algebras with non-singular derivations.
Among our examples, the reader can find Lie
algebras of derived length three with nilpotent quotients of arbitrary high
nilpotency class, and we believe that such examples
appear in the research literature for the first time. Then in Section~\ref{secJacobson} we
present an  approach to Jacobson's Theorem for solvable Lie algebras
using compatible pairs of derivations.
Compatible pairs of automorphisms were used by
Eick~\cite{Eick} to describe and compute 
the automorphism group of a solvable Lie algebra.
This method leads to the following theorem which also covers the case
of positive characteristic. The terms of the derived series of a Lie algebra
$L$ are denoted by $L^{(i)}$; see Section~\ref{chap2} for the precise definition.

\begin{teo}{\label{4.3}} Let $L$ be a finite-dimensional
  solvable Lie algebra over a field
  $\F$ of characteristic $p\geq 0$ and
  suppose that $L$ admits a non-singular derivation of finite order.
  If either $p=0$ or
   $\dim L^{(i)}/L^{(i+1)} <p$ for all $i$ then $L$ is nilpotent.
\end{teo}

The aforementioned examples in~\cite{Shalev,Mattarei}
(see also Example~\ref{mattareiex}) show that, in prime characteristic,
the condition
on 
$\dim L^{(i)}/L^{(i+1)}$ cannot be weakened in Theorem~\ref{4.3}.

In the second part of the paper, we study the class of non-nilpotent
Lie algebras that have an abelian ideal of codimension one and also admit
a non-singular derivation. Note that the examples of~\cite{Shalev,Mattarei}
belong to this class. Using the concept of $\xp$-cyclic modules introduced in
Section~\ref{prim.dec.sec}, we prove
our second main result in Section~\ref{cyclicmodchapter} that shows that the overall structure manifested by the known examples holds
in the wider class of such algebras.

 \begin{teo}{\label{4.1.7}} Let $L$ be a finite-dimensional Lie algebra of derived length $2$
   over an algebraically closed field $\F$ of prime characteristic $p$ such that
    $\dim (L/L')=1$. Let $x \in L \setminus L'$
   and assume that $x$ induces a non-singular endomorphism of finite order
   on $L'$.
   Then $L$ has a non-singular derivation of finite order if, and only if,
   $L'$ can be written as a direct sum of $\xp$-cyclic modules. 
  \end{teo}

 Theorem~\ref{4.1.7} can be interpreted as saying that finite-dimensional
 non-nilpotent Lie algebras with non-singular derivations are not as uncommon
 as the authors believed they were before starting their research on this topic. The reason
 why one does not normally see them in small dimensions
 is explained by Theorem~\ref{4.3}
 that implies that the dimension of such a solvable
 algebra is at least $\mbox{char}\,\F+1$.
 
Based on Theorem~\ref{4.1.7}, we close Section~\ref{cyclicmodchapter}
 with an application to Lie algebras that arise from representations of the Heisenberg Lie algebra and admit non-singular derivations. 

 The main results of this paper are obtained as applications of several
 theorems in linear algebra. In particular, the matrix theoretical
 Lemma~\ref{0.9} lies at the heart of the proof of
 Theorem~\ref{4.3}, while Theorem~\ref{4.1.7} relies on the concept
 of $(x,p)$-cyclic spaces which can be viewed as a restriction of the
 widely used concept of $x$-cyclic spaces.

 \section{Basic concepts}\label{chap2}
 
The symbol `$\oplus$' will be used to denote the direct sum of
algebras, while the direct sum of vector spaces will be denoted by
`$\dotplus$'. If $V$ is a vector space, then $\ennd V$ denotes the
associative algebra of endomorphisms of $V$ with the product given by
composition, while $\gl V$ denotes the Lie algebra of these
endomorphisms with the product given by the Lie bracket.
%
The Lie algebra of all
derivations of an algebra $K$ is denoted by $\der K$
and is a Lie subalgebra of $\gl K$. 

We denote by  $L^{(i)}$ the $i$-th term of the {\em derived series}
of a Lie algebra
$L$; that is, $L^{(0)}=L$ and, for $i\geq 1$, $L^{(i)}=[L^{(i-1)},L^{(i-1)}]$.
The terms $L^{(1)}$ and $L^{(2)}$ are also denoted by $L'$ and $L''$,
respectively. The smallest integer $n$ such that $L^{(n)}=0$ is called the
{\em derived length} of~$L$.

Suppose that $K$ and $I$ are Lie algebras. An {\em action}
or a {\em representation} of
$K$ on $I$ is a Lie algebra homomorphism $\psi:K\rightarrow \der
I$.
Additionally, if $I$ is an abelian Lie algebra, then $I$ is
called a {\em $K$-module}. For $k\in K$ and $v\in I$, 
the element $\psi(k)(v)$ will usually be denoted by $[k,v]$.
If $I$ is an ideal of a Lie algebra $K$, then $K$ acts on
$I$ and the endomorphism of $I$ induced by
an element $k\in K$ is denoted by $\ad I{k}$.
Hence,
for $k\in K$ and $v\in I$, $\ad Ik(v)=[k,v]$.
The corresponding representation of $K$ into $\gl I$ is denoted by
$\ad I{}$. We write
$\ad{}{}$ for the {\em adjoint representation} $\ad K{}:K\rightarrow\der K$.
%
To avoid an excess of
brackets, we use the following notation for $k,\ v\in K$: \begin{equation}
  [k,\ldots,[k,[k,v]]]=[\underbrace{k,\ldots,k,k}_{n \,
      \rm times},v]=[k^n, v].
\end{equation}\label{rightprod}
 Thus, for $v \in I$ and for $k\in K$, $(\ad
 Ik)^n(v)=[k^n,v]$ for all $n\geq 1$.

 The following lemma is well known; see~\cite[Proposition~1.3]{Strade}.

 \begin{lem}[]\label{4.1.16}
   Let $V$ be a vector space over a field $\F$. If $x,y \in \ennd V$, then
   $$[x^n,y]= \displaystyle \sum_{i=0}^{n}(-1)^i \binom{n}{i} x^{n-i}yx^i\quad \mbox{for all}\quad n \geq 1.$$
\end{lem}

Let $K$ and $I$ be Lie algebras such that $K$ acts on $I$ via the
homomorphism $\psi:K\rightarrow\der I$.  We  define the {\em
  semidirect sum} $K\sdsum_\psi I$ as the vector space $K\dotplus I$
endowed with the product operation given by
$$ [(k_1,v_1),(k_2,v_2)]=([k_1,k_2],[k_1,v_2]-[k_2,v_1]+[v_1,v_2]).
$$ When the $K$-action on $I$ is clear from the context, then we
usually suppress the homomorphism `$\psi$' from the notation and write
simply $K\sdsum I$. If $L$ is a Lie algebra, such that $L$ has an
ideal $I$, and a subalgebra $K$ in such a way that $L=K\dotplus I$,
then $L\cong K\sdsum_\psi I$ where $\psi$ is the restriction of $\ad
{I}{}$ to $K$.  In a semidirect sum $K\sdsum I$, an element $(k,v)\in
K\dotplus I$ will often be written as $k+v$. 

Suppose that $K$ and $I$ are as in the previous paragraph. 
An element
$(\alpha,\beta)\in \der K\oplus\der I$ is said to be a {\em compatible
  pair} if the linear transformation $(\alpha,\beta):K\sdsum I \to K\sdsum I$
that maps $(k,v)\mapsto (\alpha(k),\beta(v))$ is a derivation of $K\sdsum I$.
This is equivalent to the condition that
\begin{equation}{\label{compeq}}
  \beta ([k,v])=[\alpha(k),v]+[k,\beta(v)] \quad\mbox{for all}\quad k
  \in K,\ v \in I.
\end{equation}
The set of compatible pairs of $\der K\oplus\der I$ is denoted by
$\comp KI$. It is easy to check that $\comp KI$ is a Lie subalgebra of
$\der K\oplus\der I$. 

If $L$ is a Lie algebra, $x,\ y\in L$, $\delta\in\der L$, then
Leibniz's Formula~\cite[equation~(1.11)]{deGraaf} gives that
\begin{equation}{\label{leibformula}}
\delta^n([x,y])=\sum_{k=0}^n
\binom{n}{k}\left[\delta^k(x),\delta^{n-k}(y)\right]\quad \mbox{for all}\quad n\geq 1.
\end{equation}
This implies the following
lemma which will be used often in this paper.

\begin{lem}{\label{leib}} 
  Let $\F$ be a field of prime characteristic $p$.
  \begin{enumerate}
  \item If $L$ is a Lie algebra over $\F$ and $\delta \in \der L$,
    then $\delta^{p} \in \der L$.
\item
  Let $I$ and $K$ be Lie algebras such that $K$ acts on $I$ and
  let $(\alpha,\beta)\in\comp KI$. Then $(\alpha^p,\beta^p)\in\comp KI$. 
  \end{enumerate}
\end{lem}
\begin{proof}
  Equation~\eqref{leibformula} implies statement~(1).
  As explained above, $(\alpha,\beta)$ can be viewed as a derivation of
  $K\sdsum I$ by setting $(\alpha,\beta)(k,v)=(\alpha(k),\beta(v))$ for
  $k\in K$ and $v\in I$. Then
  $(\alpha,\beta)^p=(\alpha^p,\beta^p)$. By statement~(1),  $(\alpha,\beta)^p\in\der{K\sdsum I}$, and this implies~(2).
\end{proof}

An endomorphism $\alpha$ of a finite-dimensional vector space $V$ is
said to be {\em diagonalizable} if $V$ admits a basis consisting entirely
of eigenvectors of $\alpha$. In such a basis, 
the matrix of $\alpha$ is diagonal. For a non-singular linear
transformation $\alpha$ of finite order, let $|\alpha|$ denote the
order of $\alpha$.  The following lemma is well known.

\begin{lem}{\label{NonSDiag}}  
 Suppose that $L$ is a finite-dimensional Lie algebra over a field
 $\F$ of characteristic $ p \geq 0$ and let $\delta$ be a non-singular
 derivation of $L$ with finite order. Then there exists an algebraic
 extension $\K$ of $\F$ such that one of the following is valid:
\begin{enumerate}
\item $p = 0$ and $\delta$ is diagonalizable over $\K$;
\item $p$ is a prime, $|\delta|=n{p^t}$ with $p \nmid n$, and
  $\delta^{p^t}$ is a non-singular derivation that is diagonalizable
  over $\K$.
\end{enumerate}
\end{lem}

 Let $L$ be a Lie algebra over a field $\F$ and let $A$ be an abelian group. Suppose that
 for all $a\in A$, $L_a$ is a linear subspace of $L$ such that
 \[
 L=\bigdotplus_{a\in A} L_a
 \]
 and
 $[L_a,L_b]\leq L_{a+b}$ holds for all $a,\ b\in A$. Then $L$ is said to
 be an {\em $A$-graded Lie algebra}. The elements in the subspaces
 $L_a$ are said to be {\em homogeneous}. 
 There is a strong connection between gradings and diagonalizable
 derivations on a Lie algebra $L$.  Indeed, if $\delta\in\der L$ is
 a diagonalizable derivation of $L$,
 and, for $\alpha\in\F$, $L_\alpha$  is the $\alpha$-eigenspace of $\delta$
 (taking $L_\alpha=0$ when $\alpha$ is not an eigenvalue of $\delta$), then
 $$
 L=\bigdotplus_{\alpha\in\F}L_\alpha
 $$
 is a grading on $L$ with respect to the additive group of $\F$.
 Conversely, if $L$ is $\F$-graded, then we  define a  derivation
 $\delta$ on $L$ by setting $\delta(x)=\alpha x$ for all $x$ that lie
 in the homogeneous component $L_\alpha$. If $L_0=0$, then the derivation
 $\delta$ is non-singular.

 The following
 result is a bit more general than~\cite[Corollary 5.2.7]{Leedham}, but
 can be proved using the same argument; see also~\cite[Lemma~2.1]{Shalev}.

  \begin{lem}{\label{4.2.4}} Let
   $L$ be a finite-dimensional Lie algebra which is graded by some abelian group. Suppose that $L$ satisfies the graded $n$-Engel condition for some $n\geq 1$;
   that is $[x^n,y]=0$ for all homogeneous elements $x,\ y \in L$. Then $L$ is nilpotent.
\end{lem}



For the following classical theorem of Jacobson; see~\cite[Theorem~3]{Jacobson}.
 
\begin{teo}\label{Jacobsontheo} Let $L$ be a finite-dimensional Lie algebra over a field of characteristic $0$ and suppose that there exists a subalgebra $D$ of the algebra of derivations of $L$ such that
\begin{enumerate}
 \item $D$ is nilpotent;
 \item if  $x \in L$ such that $\delta(x)=0$ for all $\delta
   \in D$ then $x=0$.
\end{enumerate} Then $L$ is nilpotent. In particular,
if $L$ admits a non-singular derivation, then $L$ is nilpotent.
\end{teo}

In  Theorem \ref{Jacobsontheo}, the hypothesis of zero
characteristic is essential to prove that every element in a
homogeneous component is nilpotent. 
Theorem \ref{Jacobsontheo} fails to hold in prime characteristic as
shown by the following example.

\begin{ex}[{\cite[Theorem 2.1]{Mattarei}}]\label{mattareiex}
  Let $V$ be a $p$-dimensional vector space over a field $\F$ of characteristic $p>0$ with basis $v_0, \ldots, v_{p-1}$. Define the linear map $x \in \gl V$ by $x(v_i)=v_{i+1}$ for $0 \leq i \leq p-2$ and $x(v_{p-1})=v_0$. Let $K$ be the abelian Lie algebra generated by $x$
  and consider $V$ as a $K$-module.  Set $L=K \sdsum V$. Then $L$ is a solvable non-nilpotent Lie algebra of derived length 2. Suppose that
  $\alpha,\ \beta \in \F\setminus\{0\}$ such that $\alpha \neq \gamma \beta $
  for all $\gamma \in \F_p$
  (such elements exist if $|\F|\geq p^2$). Define the
  linear map $\delta:L \to L$ as
  $$\delta: \left\lbrace\begin{array}{lc}
 x \mapsto \alpha x; \\
 v_i \mapsto (\beta +(i+1)\alpha)v_i & \mbox{for $0 \leq i \leq p-1$.}
\end{array} \right.$$ 
Then $\delta$ is a non-singular derivation of $L$.
\end{ex}

It is known that the converse of Jacobson's Theorem is not valid.
 Dixmier
 and Lister \cite{Dixmier} presented finite-dimensional
 nilpotent Lie algebras admitting
 only nilpotent derivations.

 We end this section with a result concerning matrices.
 For a field $\F$ and for a natural number $n$, let ${\sf M}(n,\F)$
 denote the set of $n\times n$ matrices over $\F$. 

\begin{lem}{\label{0.9}} Let $\F$ be a field of characteristic
  $p \geq 0$, let $n$ be a natural number and let $A,\ B,\ C$ be
  $n\times n$ matrices over $\F$. Assume also that
  either $p=0$ or $n<p$.
  If $[A,B]=C+\lambda B$ for some $\lambda \in \F$ and $[B,C]=0$, then $[A,B^r]=rB^{r-1} C+\lambda r B^r$ for all $r \geq 1$. In particular, if $\lambda \neq 0$ and  $C$ is nilpotent, then $B$ is nilpotent.
\end{lem}
\begin{proof}
  The lemma can be proved following the proof of~\cite[Fact 3.17.13]{Bernstein} and
  making the necessary adaptations in the case of prime characteristic.
\end{proof}

\section{Examples of non-nilpotent Lie algebras with non-singular derivations}
\label{sec:examples}

In this section we present some examples of non-nilpotent and solvable
Lie algebras that admit non-singular derivations, including examples
of derived length three and examples with nilpotent quotients of arbitrary
high nilpotency class.

Our first example is a Lie algebra of the form $L=K\sdsum I$ where $K$ is a
nilpotent Lie algebra of maximal class and $I$ is a faithful $K$-module.
  Let $\F$ be a field of prime characteristic
  $p$ and
  let $I$ be an $\F$-vector space of dimension $2p$.
Suppose that $\{v_1,\ldots,v_{2p}\}$ is a basis of $I$ and define the 
 the elements $x,\ y \in \ennd I$ with the following rules
\begin{eqnarray*}
x&:& v_{1} \mapsto v_{2}, \, v_{2} \mapsto v_{3}, 
\ldots v_{p-1} \mapsto v_{p}, \, v_{p} \mapsto v_{1},\\
&& v_{p+1} \mapsto v_{p+2}, \, v_{p+2} \mapsto v_{p+3}, 
\ldots v_{2p-1} \mapsto v_{2p}, \, v_{2p} \mapsto v_{p+1};\\
y&:&v_{p+1} \mapsto v_1 \quad\mbox{and} \quad v_i \mapsto 0 \quad\mbox{if}\quad i
  \neq p+1.
\end{eqnarray*}

\begin{prop}\label{6.2.3}
  Let  $K$ be the Lie subalgebra of $\gl I$
  generated by $x$ and $y$ as defined above and set $L=K\sdsum I$.
  Then the following hold.
  \begin{enumerate}
  \item  $K$ is a nilpotent Lie algebra of  dimension $p+1$
    and nilpotency class $p$.
  \item $L=K\sdsum I$ is not nilpotent and
    has derived length~$3$.
  \item If $|\F|\geq p^2$, then $L$ admits a non-singular derivation
    of finite order.
  \end{enumerate}
\end{prop}

\begin{proof}
(1)  Let $[w_1,w_2,\ldots,w_r]$ be a right-normed product in $K$ such that $w_j
  \in \{x,y\}$ for $j\in\{1,\ldots,r\}$. As is well known, such products generate $K$.
  We claim that $[w_1,w_2,\ldots,w_r]=0$ unless $[w_1,w_2,\ldots,w_r]$ is
  either of the form $[x,\ldots,x,y]$ or of the form $[x,\ldots,x,y,x]$. 
  Let $I_1$ be
  the vector subspace generated by $\{v_1, \ldots, v_p\}$. By the
  definition of $x,\ y \in \ennd I$,
  \begin{equation}\label{7.1}
    x(I_1)=I_1, \quad y(I)\leq I_1  \quad \mbox{and }
 \quad y(I_1)=0.
 \end{equation}
   Denoting the symmetric group of degree $r$ by $S_r$,
   there exist coefficients $c_\sigma\in\F$ for all $\sigma\in S_r$ such that
\begin{equation}\label{7.2}
[w_1,w_2,\ldots,w_r]= \sum_{\sigma \in S_r}^s c_\sigma w_{1\sigma}w_{2\sigma}
\cdots w_{r\sigma}.
\end{equation} 
If $y$ appears more than once in the sequence $w_1,w_2,\ldots,w_r$, then
$w_{1\sigma}w_{2\sigma} \ldots w_{r\sigma}=0$ holds for all $\sigma$, by~\eqref{7.1}.
Suppose now that $y$ appears exactly once in the sequence
  $w_1,w_2,\ldots,w_r$ and suppose that $w_j=y$. If $j \leq r-2$, then
 $[w_1,w_2,\ldots,w_r]=[x,x,\ldots,y, \ldots, x,x]=0$.
 Thus the only possibilities for a non-zero right-normed product are
 $[x,\ldots,x,y]$ and $[x,\ldots,x,y,x]$ as claimed.

  By Lemma~\ref{4.1.16},
  \begin{eqnarray*}
  [x^{p-1},y]v_{p+1}  &=& 
  \sum_{i=0}^{p-1}(-1)^i\binom{p-1}{i} x^{p-1-i}yx^i v_{p+1}  \\   &=&
  \displaystyle   \sum_{i=0}^{p-1}(-1)^i\binom{p-1}{i}
  x^{p-1-i}yv_{p+1+i}    =   x^{p-1}yv_{p+1}   =   x^{p-1}v_{1}   = 
  v_p. 
  \end{eqnarray*}
    Therefore $[x^{p-1},y]\neq 0$.
    On the other hand, $x^p$ acts as
  identity on $I$. Further, by Lemma~\ref{4.1.16},
  $[x^p,y]=x^py-yx^p=0$.
  Thus, noting that
 $[x,\ldots,x,y,x]=-[x,\ldots,x,x,y]$ and interpreting
  $[x^0,y]$ as $y$, the set $\{x\}\cup\{[x^n,y]\mid 0\leq n\leq p-1\}$
  is a linear generating set for $K$. This implies that $K$ has nilpotency
  class $p$ and dimension at most $p+1$. On the other hand, a nilpotent Lie
  algebra of nilpotency class $p$ has dimension at least $p+1$. Thus
  $\dim K=p+1$.

  (2) 
  The derived series of $L$ is
  $$
  L> \left<[x,y],[x,x,y],\ldots,[x^{p-1},y],I\right>> \left<v_1,\ldots,v_p\right>> 0.
  $$
  Hence $L$ is solvable of derived length 3. On the other hand,
  as $x$ induces a non-singular transformation on $I$,
  we have that $[x^n,I]=I$ for all $n\geq 1$, and so
  $L$ is not nilpotent.


  (3) We claim that $L$ admits a grading with respect to the additive group of
  $\F$. First, $I$ can be graded by assigning $v_{kp+i}$ degree $b-ka+i-1$.
  Now the linear transformations $x$ and $y$ preserve this grading in the sense
  that $x(I_{\alpha})\subseteq I_{\alpha+1}$ and $y(I_\alpha)\subseteq I_{\alpha+a}$
  for all $\alpha\in\F$.
  Consequently, the Lie algebra $K$, generated by $x$ and $y$, is $\F$-graded;
  the degrees  of $x$ and $y$ are $1$ and $a$, respectively, while the degree of
  $z_j$ is $a+j$ for $j\in\{1,\ldots,p-1\}$. Thus the semidirect sum $L$
  is graded over the additive group of the field $\F$, and so
  the linear map $\delta$, which multiplies each homogeneous element by its degree, is a derivation (see the discussion after Lemma~\ref{NonSDiag}).
  By the choice of $a,b\in\F$, the homogeneous component $L_0$ is trivial, and
  so $\delta$ is non-singular.
  \end{proof}

Our next example involves the  Heisenberg Lie algebra, which is
the 3-dimensional Lie algebra
$H=\left<x,y,z\right>$ over a field $\F$ with the multiplication
$[x,y]=z$ and $[x,z]=[y,z]=0$. The Lie algebra $H$ is nilpotent with
nilpotency class~2. Suppose that $\Char \F$ is prime and
let $I$ be a vector space of dimension $2p$
over $\F$ with basis $\{v_1,\ldots,v_{2p}\}$.
Define the endomorphisms $x$, $y$, and $z$ of $I$ as follows:
\begin{eqnarray*}
  x&:&v_1\mapsto v_2,\ v_2\mapsto v_3, \ldots,v_{p-1}\mapsto v_{p},\ v_{p}\mapsto v_1,\\
  &&
  v_{p+1}\mapsto v_{p+2},\ v_{p+2}\mapsto v_{p+3}, \ldots,v_{2p-1}\mapsto v_{2p},\ v_{2p}\mapsto v_{p+1};\\  
  y&:&v_i\mapsto 0\quad\mbox{for all}\quad i\in\{1,\ldots,p\},\\
  && v_{p+1} \mapsto 0,\ v_{p+2}\mapsto v_2,\ v_{p+3}\mapsto 2v_3,\ldots,
  v_{2p}\mapsto (p-1)v_{p}; \\
z&:& v_i\mapsto 0 \quad\mbox{for all}\quad i\in\{1,\ldots,p\},\\
&&v_{p+1}\mapsto v_{2},\ v_{p+2}\mapsto v_3,\ldots,v_{2p-1}\mapsto v_{p},\
v_{2p}\mapsto v_1.
\end{eqnarray*}
Easy computation shows that $[x,y](v_i)=zv_i$ for all $i$
and that $[x,z]=[y,z]=0$. 
 Hence the Lie algebra $H=\left<x,y,z\right>$ is a Heisenberg Lie
 algebra over $\F$. Set $L=H\sdsum I$.

 \begin{prop}
   The Lie algebra $L$ is solvable of derived length $3$, but not
   nilpotent. If $|\F|\geq p^2$, then $L$ admits a non-singular derivation.
   \end{prop}
 \begin{proof}
   The derived series of $L$ is
   $$
   L>\left<z,I\right>>\left<v_1,\ldots,v_p\right>>0,
   $$
   and so the derived length of $L$ is~3. Further,
   $[x^n,I]=I$ for all $n\geq 1$, and so $L$ is not nilpotent.
    Suppose that $|\F| \geq p^2$. There
 are $a,\ b \in (\F \setminus  \F_p)$ such that $b - a\not\in\F_p$.

 Following the steps seen in the proof of Proposition \ref{6.2.3}, setting the degree $b-ka+i-2$ for each $v_{kp+i}$, $I$ is a graded Lie algebra. The linear maps $x$, $y$ and $z$ are graded on $I$, with degrees $1$, $a$ and $1+a$, respectively. Thus, $L$ is an $\F$-graded Lie algebra with $L_0=0$ and hence there is a corresponding non-singular
derivation $\delta$ as explained after Lemma~\ref{NonSDiag}.
\end{proof}

 \section{Jacobson's Theorem through compatible pairs}{\label{secJacobson}}

The aim of this section is to prove Theorem~\ref{4.3}. For a Lie algebra $K$ and for a $K$-module $I$, we let, as in Section~\ref{chap2}, $\comp KI$ denote the set of compatible pairs in $\der
K\oplus\der I$.  Using the representation $\psi:K\rightarrow\der I$,
we can rewrite equation
(\ref{compeq}) as 
\begin{equation}{\label{compcomu}}
  [\beta,\psi(k)]=\psi (\alpha(k)) \mbox\quad \mbox{for all} \quad k
  \in K.
\end{equation}
Thus $(\alpha,\beta)\in\comp KI$ if and only if~\eqref{compcomu} is valid.
Letting $\ad{}{}:\der I\to \der I$ denote the  adjoint representation
of $I$, equation~\eqref{compcomu} can be rewritten as
\begin{equation}{\label{compcomu2}}
\ad{}{\beta}\psi(k)=\psi (\alpha(k))\quad \mbox{for all}\quad k \in
K.
\end{equation}
Therefore, $(\alpha,\beta) \in\comp KI$ if, and only if, the following
diagram commutes:
$$ \xymatrix{K \ar[d]^{ \alpha} \ar[r]^{\psi}
  \ar@{}[dr]|{\circlearrowright} & {\der I} \ar[d]^{\ad{}{\beta}} \\ K
  \ar[r]^{\psi} & {\der I}.} $$

\begin{teo}{\label{1.8}} Let $K$ be a finite-dimensional solvable Lie algebra
  over an algebra\-ically closed field $\F$ of characteristic $p \geq 0$
  and let $I$ be a $K$-module under the representation
  $\psi:K \to \der I$. Let $(\alpha, \beta) \in \comp K I$ such that $\alpha$ is non-singular with finite order. If either $p=0$ or
  $\dim I < p$, then $\psi(k)$ is nilpotent for all $k \in K$. 
\end{teo}

\begin{proof}
  If $\Char \F=0$, then $\alpha$ is diagonalizable by Lemma~\ref{NonSDiag}.
  If $\Char \F$ is prime, then the same lemma implies, for a suitable $t$,
  that
  $\alpha^{p^t}$ is a non-singular
  and diagonalizable derivation of $K$ such that $(\alpha^{p^t},\beta^{p^t})\in
  \comp KI$. Hence, by Lemma~\ref{leib}, we may assume without loss of generality that $(\alpha,\beta)\in\comp KI$ with $\alpha$ non-singular, diagonalizable,
  and of finite order.
  Let $x_1,\ldots,x_s$ be a basis of $K$ such that $\alpha(x_i)=\lambda_i x_i$. Let $B$ be a basis for $I$ and, for all $a\in\gl I$, denote by
  $\llbracket a \rrbracket $ the matrix
  of $a$ in $B$. Then, by equation \eqref{compcomu},
  $$
  [\llbracket\beta\rrbracket,\llbracket\psi(x_i)\rrbracket]= \llbracket \psi(\alpha(x_i))\rrbracket=\lambda_i \llbracket \psi(x_i)\rrbracket.$$  
 Applying Lemma~\ref{0.9} with $A=
 \llbracket \beta \rrbracket$, $B=\llbracket\psi(x_{i})\rrbracket$,
 $C=0$ and  $\lambda=\lambda_i$ (which is non-zero) we conclude that
 $\llbracket\psi(x_i) \rrbracket$ is nilpotent for $1 \leq i \leq
 s$. Now, Lie's Theorem is valid for the solvable Lie algebra $K$
 also in the case when $\Char\F>0$ and $\dim I<p$ (see~\cite[Section~4.1 and
   Exercise~2 on page~20]{Humphreys}).
 Thus there is a basis of $I$ such that
 $\psi(K)$ is upper
 triangular. Working in this basis, since
 $\llbracket\psi(x_i)\rrbracket$ is nilpotent and upper triangular, it
 must be strictly upper triangular (that is, with zeros in the
 diagonal) for all $i$. Therefore $\llbracket \psi(k) \rrbracket$
 is strictly upper triangular, and hence nilpotent, for all $k\in K$.
\end{proof} 

From the previous theorem, we obtain Theorem~\ref{4.3} which can
be viewed as a  version of
Jacobson's Theorem on non-singular derivations for
solvable Lie algebras which is  valid also
in prime characteristic.

\begin{proof}[Proof of Theorem~\ref{4.3}] Since the solvability of $L$ and $\dim L^{(i)}/L^{(i+1)}$
  are invariant under  extensions of the ground field,
  we may assume that $\F$ is algebraically closed.
  We prove the statement by induction on the derived length $k$ of $L$.
  When $k=1$, then $L$ is clearly nilpotent. Suppose that the result holds for Lie algebras of derived length $k$ and assume that $L$ has derived length $k+1.$ Then $I=L^{(k)}$ is an abelian ideal of $L$. Setting $K=L/I$, we have that $K$ acts on $I$ by the adjoint representation $\ad {I}{}:K \to \der I$. Further, since the terms of the  derived series are invariant under derivations, a non-singular derivation $\delta \in \der L$ gives rise to a compatible pair $(\alpha,\beta) \in \comp KI$ where $\alpha$ is
  the derivation induced on $K$ by $\delta$ and $\beta$ is the restriction
  of $\delta$ to $I$. Since $\delta$ is non-singular, so are $\alpha \in \der K$ and $\beta\in \der I$. Note that $K$ is solvable of derived length $k$ and $K^{(i)}/K^{(i+1)} \cong L^{(i)}/L^{(i+1)}$ for all $i\leq k-1$. Hence the induction hypothesis is valid for $K$ and we obtain that $K$ is nilpotent. In addition,
  Theorem \ref{1.8} implies that $\ad{I}{k}$ is nilpotent for all $k \in K$. Therefore, $L/I$ is nilpotent and $\ad{I}{x}:I \to I$ is nilpotent for all $x \in L$. It follows from~\cite[Exercise~10, page~14]{Humphreys} that $L$ is nilpotent. 
 \end{proof}

\section{Primary decomposition and $p$-cyclic spaces}\label{prim.dec.sec}
In this section we review a primary decomposition concept and define
$\xp$-cyclic spaces that appear in Theorem~\ref{4.1.7}.
 Let $V$ be a finite-dimensional vector space over a field $\F$ and let $x
\in \ennd V$. Let $q \in \F[t]$ be a univariate polynomial and define
 \begin{equation}{\label{linear.01}}
 V_0(q(x))=\{v \in V \mid  \mbox{ there is } m>0  \mbox{ such that
 } q(x)^mv=0\}.
 \end{equation}
 The set $V_0(q(x))$ is a vector subspace of $V$ which is invariant
 under $x$.
 Let  $q_x$ be the minimal polynomial of $x$
 and suppose that  $q_x=q_1^{k_1} \cdots q_r^{k_r}$ is the
 factorization of $q_x$ into monic irreducible factors, such that
 $q_i \neq q_j$ for $1 \leq i< j \leq
 r$. Then $V$ decomposes as a direct sum of subspaces $$V=V_0(q_1(x))
 \dotplus \cdots \dotplus V_0(q_r(x)),$$ with each space $V_0(q_i(x))$
 being invariant under $x$. Furthermore, the minimal polynomial of the
 restriction of $x$ to $V_0(q_i(x))$ is $q_i^{k_i}$. A proof of this
 result can be found in \cite[Lemma A.2.2]{deGraaf}.   
 
We can generalize this decomposition to subalgebras of $\gl V$
generated by more than one element. The following definition
was stated in~\cite[Definition~3.1.1]{deGraaf}.

\begin{df}{\label{linear.02}} Let $V$ be a finite-dimensional vector space over a field $\F$ and let $K \leq \gl V$ be a subalgebra. A decomposition $V = V_1 \dotplus \cdots \dotplus V_s$ of $V$ into $K$-submodules $V_i$ is said to be \textit{primary} if the minimal polynomial of the restriction of $x$ to $V_i$ is a power of an irreducible polynomial for all $x \in K$ and $1 \leq i \leq s$. The subspaces $V_i$ are called \textit{primary components}.
 If for any two components $V_i$ and $V_j$ $(i \neq j)$, there is an
 $x \in K$ such that the minimal polynomials of the restrictions of
 $x$ to $V_i$ and $V_j$ are powers of different irreducible
 polynomials, then the decomposition is called \textit{collected}. 
  \end{df}
  
  In general, a $K$-module $V$ may not have a primary (or collected
  primary) decomposition into $K$-submodules, but such a decomposition is
  guaranteed to exist if  $K$, as a subalgebra of $\gl V$, is
  nilpotent; see~\cite[Theorem 3.1.10]{deGraaf}.

  \begin{lem}[{\cite[Proposition 3.1.7]{deGraaf}}]{\label{linear.05}}
    Suppose that $V$ is a finite-dimensional vector space over a field $\F$,
    let $x,\ y \in \gl V$ and let $q \in \F[t]$ be a polynomial. Suppose that $[x^n,y]=0$ for some $n \geq 1$. Then $V_0(q(x))$ is invariant under $y$.
\end{lem}

\begin{lem}\label{5.3} Let $K$ be a nilpotent Lie algebra over a field $\F$ of characteristic $p>0$ and let $I$ be a finite-dimensional $K$-module.
  Let $L=K \sdsum I$, $x \in K$ and $q(t)=t-a$ with some $a \in \F$.
  Let $\delta \in \der L$ such that $\delta(I)\leq I$ and
  $\delta(x)=b x$ for some $b\in\F$. Then $ I_0(q(x))$ is $\delta$-invariant.
\end{lem}
\begin{proof}   
 Suppose that
 $w \in I_0(q(x))$; that is, there is $m>0$ such that $(x-a\cdot \id)^{p^m} \cdot  w=0$. As $\Char \F=p$, we have 
\begin{equation}{\label{5.2}}
0=(x-a \cdot \id)^{p^m} \cdot  w=(x^{p^m}-a^{p^m}\cdot \id) \cdot
  w=x^{p^m}\cdot w-a^{p^m}w.
\end{equation}
As $\delta \in \der L$, using the right-normed convention introduced
in equation \eqref{rightprod},
\begin{eqnarray}{\label{5.1}}
\delta(x^{p^m} \cdot w)  & = &\delta([x^{p^m},w]) \nonumber \\ & =
&[\delta(x),\ldots,x,w]+[x,\delta(x),\ldots,x,w]+\cdots+[x,\ldots,x,\delta(w)]
\nonumber \\ & = &
          [ax,\ldots,x,w]+[x,ax,\ldots,x,w]+\cdots+[x,\ldots,x,\delta(w)]
          \nonumber\\ & = & {p^m} \cdot a \cdot
                      [x^{p^m},w]+[x^{p^m},\delta(w)]\nonumber \\ & =
                      & x^{p^m} \cdot \delta(w).
\end{eqnarray} 
 Combining \eqref{5.1} and \eqref{5.2} we
 obtain $$0=\delta(0)=\delta(x^{p^m} \cdot w-a^{p^m}w)=\delta(x^{p^m}
 \cdot w) - a^{p^m}\delta(w)=x^{p^m} \cdot \delta(w)-
 a^{p^m}\delta(w).$$ Hence, $$(x - a\cdot\id)^{p^m} \cdot \delta(w)=x^{p^m}
 \cdot \delta(w)-a^{p^m}\delta(w) =0$$ and $\delta(w) \in I_0(q(x)).$
\end{proof}

Let $V$ be a finite-dimensional vector space over a field $\F$ and
$x \in \ennd V$. The vector space $V$ is said to be \textit{$x$-cyclic} if
there is $v \in V$ such that $\{x^k(v)\mid k\geq 0\}$ is a generating set for
$V$. As is well known, $V$ is $x$-cyclic if and only if
the degree of the minimal polynomial of the endomorphism $x$
coincides with $\dim V$, which amounts to saying that the minimal polynomial
of $x$ is equal to its characteristic polynomial.
It is well known that if $V$ is a vector space and $x\in\ennd V$,
then $V$ can be decomposed as a direct sum of $x$-cyclic subspaces;
see~\cite[Theorem~7.6]{Roman} and~\cite[Theorem 1.5.8 and
   Corollary 1.5.14]{Robbiano} 
 for more details.




 In the non-nilpotent Lie algebra with non-singular derivation given by
 Example~\ref{mattareiex}, the vector space $V$ is $x$-cyclic, but it also
 satisfies some stronger conditions. This motivates the following concept.
 
 \begin{df}{\label{4.1.8}} Let $V$ be a vector space over a
   field $\F$ of prime characteristic $p$ and let $x \in \gl V$.
   The vector space $V$ is \textit{$\xp$-cyclic} if the following hold:
\begin{enumerate}
\item $V$ is $x$-cyclic;
\item $q_x(t)\in\F[t^p]$ and $q_x(0)\neq 0$.
\end{enumerate}
\end{df}

 Thus if $V$ is $\xp$-cyclic, as in Definition~\ref{4.1.8}, then
 the minimal polynomial $q_x(t)$ is of the form
 $c_0+c_1t^p+ \cdots+c_{n-1}t^{(n-1)p}+t^{np}$ with $c_0\neq 0$. 
 Hence, over a perfect field $\F$,
 $q_x(t)$ can be written as the $p$-th power of a polynomial $q_0(t)$.

 \section{Lie algebras with an abelian ideal of codimension one}
         {\label{cyclicmodchapter}}

 In this section we consider non-singular derivations of a Lie
 algebra $L$ such that the derived subalgebra $L'$ is abelian
 and has codimension one.
 If $x\in L\setminus L'$, then $L$ can be written as
 $L=K\sdsum L'$ where $K=\left<x\right>$ and the action of $K$ on $L'$
 is given by the restriction of $\ad {L'}{}$ to $K$. The Lie algebra
 $L$ presented in Example~\ref{mattareiex} satisfies this condition and
 we observed at the end of Section~\ref{prim.dec.sec}
 that $L'$ is an $\xp$-cyclic
 module. The objective of this section is to prove Theorem~\ref{4.1.7}
 stated in the Introduction which
 shows that this phenomenon can be observed in a more general case.

 We start with the easier direction of Theorem~\ref{4.1.7}.

 \begin{lem}{\label{4.1.9}} Let $K$ be a one-dimensional Lie algebra over an algebraically closed field $\F$ of prime characteristic $p$
   with $K=\langle x \rangle$ and let $I_1, I_2, \ldots, I_s$ be
   $\xp$-cyclic $K$-modules. Then the Lie algebra $L$ given by the semidirect sum
   $$
   L=K \sdsum (I_1 \dotplus I_2 \dotplus \cdots \dotplus I_s)
   $$
   has a non-singular derivation with $sp+1$ distinct eigenvalues.
\end{lem}
 \begin{proof} As $\F$ is algebraically closed, we can choose $b,a_1, \ldots,a_s \in \F$ such that $a_jb^{-1} \notin \F_p,$ for all $ 1 \leq j \leq s$, and
   \begin{equation}\label{eigenset}
     |\{a_j+ib \mid 1 \leq j \leq s \mbox{ and } 0 \leq i \leq p-1\}|=ps.
     \end{equation}
     By assumption, $I_j$ is $\xp$-cyclic, for $1 \leq j \leq s$, and so there is a basis $$B_j=\{v_0^j, v_1^j, \ldots, v_{r_jp-1}^j\}$$ of $I_j$ such that the matrix of $x$ in
   $B_j$ is the companion matrix of the minimal polynomial
   $-c_{0,j}-c_{p,j}t^p-\cdots-c_{(r_j-1)p,j}t^{(r_j-1)p}+t^{r_jp}$
   with $c_{i,j}\in\F$.
   Let $1 \leq j \leq s$. Then
\begin{eqnarray*}
\left[x,v_i^j\right]       & = & v_{i+1}^j \quad\mbox{for}\quad 0 \leq i < r_jp-1\quad\mbox{and}\\
\left[x, v_{r_jp-1}^j\right] & = & \sum_{i=0}^{r_j-1}c_{ip,j}v_{ip}^j. 
\end{eqnarray*}
Define an $\F$-grading on $L$ by setting
the degrees of $x$ and the $v_i^j$, for all possible $i$ and $j$,
to be $b$ and $a_j+ib$, respectively.
The fact that this assignment defines a grading on $L$ follows
easily from the last two displayed equations.
By the discussion after Lemma~\ref{NonSDiag}, there is a corresponding
derivation $\delta$ on $L$. By the choice of $a,b\in\F$, the homogeneous
component $L_0$ is trivial, and so $\delta$ is non-singular. The eigenvalues
of $\delta$ are $b$ and the elements of the set in~\eqref{eigenset}, and hence
$\delta$ has $sp+1$ distinct eigenvalues, as claimed.
\end{proof}

 The following lemma will be used as a first step in the proof
 of Theorem~\ref{4.1.7}.

\begin{lem}\label{4.1.5}
  Let $K=\langle x \rangle$ be a one-dimensional
  Lie algebra over  $\F$ of prime characteristic $p.$ Let $I$ be a
  finite-dimensional $K$-module such that $x$ induces an invertible
  endomorphism of $I$ and set $L=K \sdsum I$. Assume that $\delta$ is
  a non-singular derivation of $L$ such that $\delta(x)=x$ and
  $\delta|_I$ is diagonalizable. If $v\in I$ is an
  eigenvector of $\delta$, then the $K$-submodule $\langle v\rangle_K$
  is $(x,p)$-cyclic.  
\end{lem}
\begin{proof}  Define the sequence $v_0=v$
  and $v_{i+1}=[x,v_{i}]$ for $i \geq 0$. Then the set $\{v_0,v_1, \ldots, \}$ generates $\langle v\rangle_K$ and $\langle v\rangle_K$ is $x$-cyclic. As $I$ has finite dimension, there is a $k >0$ such that $\{v_0, v_1, \ldots, v_{k-1}\}$ is linearly independent and $\{v_0, v_1, \ldots, v_{k}\}$ is linearly dependent.
  Since $x$ and $v_0$ are  $\delta$-eigenvectors, easy induction
  shows that  each $v_i$ is an eigenvector of $\delta$ associated to the eigenvalue $a+i$, where $a$ is the eigenvalue 
associated with $v$. Note that  $a, a+1, \ldots , a+(p-1)$ are
distinct eigenvalues, and hence the set $\{v_0,v_1, \ldots, v_{p-1}\}$ is
linearly independent. In particular, $k \geq p$. If the eigenvectors $v_i$ and
$v_j$ are associated with eigenvalues $a+i$ and $a+j$, then $v_i$ and
$v_j$ are associated with the same eigenvalue if and only if
$i\equiv j \pmod p$. Suppose that $k=rp+t$ with $0 \leq t \leq p-1$. Since
a set of eigenvectors associated to pairwise distinct eigenvalues is
linearly independent, the eigenvector $v_k$ must be a linear
combination of the eigenvectors $v_i$ with $i \leq k-1$ that have the
same eigenvalue as $v_k$, which is $a+t$. Hence,
   \begin{equation}{\label{4.1.1}}
v_k=c_0v_t+c_1v_{p+t}+c_2v_{2p+t}+\cdots +c_{r-1}v_{(r-1)p+t}.
\end{equation}  If $t\neq 0$, then we can replace every $v_i$ by $[x,v_{i-1}]$ in equation \eqref{4.1.1} and obtain  
\begin{equation}
[x,v_{k-1}]=c_0[x,v_{t-1}]+c_1[x,v_{p+t-1}]+\cdots
+c_{r-1}[x,v_{(r-1)p+t-1}].
\end{equation}
Since  $x$ induces an injective endomorphism on $I$,
$$v_{k-1}=c_0v_{t-1}+c_1v_{p+t-1}+c_2v_{2p+t-1}+\cdots
+c_{r-1}v_{(r-1)p+t-1}.$$ This contradicts to the assumption that
$\{v_0, v_1, \ldots, v_{k-1}\}$ is linearly independent. Thus, $t=0$
and $k=rp$. Equation \eqref{4.1.1} implies also that
$v_k=c_0v_0+c_1v_p+ \cdots + c_{r-1}v_{(r-1)p}$ and so, the
characteristic polynomial of $x$ is
 $$q_x(t)=t^{rp}-c_{(r-1)}t^{(r-1)p}- \cdots -
c_{2}t^{2p}-c_1t^p-c_0.$$ If $c_0=0$, then replacing each $v_i$ with
$[x,v_{i-1}]$ as above implies that the set $\{ v_0, v_1, \ldots, v_{k-1}\}$
is linearly dependent. Therefore $c_0 \neq 0$. As $\langle v\rangle_K$
is $x$-cyclic, the minimal polynomial of the restriction of $x$
to $\langle v\rangle_K$ is $q_x$ and $\langle v\rangle_K$ is $\xp$-cyclic.
\end{proof}

For an endomorphism $x$ of a vector space $I$, 
we denoted the  minimal polynomial of $x$ by $q_x$.
When we want to emphasize the
domain of $x$, we use the notation $q_{x,I}$. If $v\in I$, then
$q_{x,v}$ denotes the minimal polynomial of $x$ with respect to
$v$. That is, $q_{x,v}$ is the smallest degree, non-zero, monic
polynomial such that $q_{x,v}(x)(v)=0$. It is well known that
$q_{x,v}\mid q_{x,I}$ for all $v\in I$. The proof of the following
theorem was inspired by the proof of Theorem 6.6 in \cite{Robertson}. 

\begin{lem}\label{4.1.19} Let $K=\langle x \rangle$ be a Lie algebra of dimension $1$ over an algebraically closed field $\F$ of characteristic $p>0.$ Let $I$ be a finite-dimensional $K$-module such that $x$ induces an invertible endomorphism of finite order on  $I$ and set $L=K \sdsum I$. Assume that $\delta$ is a non-singular derivation of $L$ such that $\delta(x)=x$ and $\delta|_I$ is diagonalizable. Assume, further, that $q_{x,I}(t)=(t-\lambda)^m$ with some $\lambda\in\F$ and $m\geq 1$ and that the $\delta$-eigenvalues on $I$ are $a,a+1,\ldots,a+p-1$
  with some $a\in\F$. Then $I$ is the direct sum of $(x,p)$-cyclic
  subspaces, each of which is generated by a $\delta$-eigenvector.
\end{lem}
\begin{proof}
  We prove this lemma by induction on $\dim I$. By Lemma \ref{4.1.5},
  $\dim I\geq p$, and so the base case of the induction is when $\dim
  I=p$. In this case, if $v\in I$ is a $\delta$-eigenvector, then
  $\left<v\right>_K$ is $(x,p)$-cyclic of dimension greater than or
  equal to $p$, and hence $I=\left<v\right>_K$. Thus the lemma is
  valid when $\dim I=p$.

  Suppose now that $\dim I\geq p+1$ and that the lemma is valid for
  spaces of dimension less than $\dim I$. By our conditions,
  $I=E_a\dotplus\cdots\dotplus E_{a+p-1}$ where $E_b$ denotes the
  $b$-eigenspace of $\delta$ in $I$. Since $\bigcup_b E_b$ generates
  $I$ as a vector space, there is some eigenvector $v_0\in I$ such
  that $q_{x,v_0}(t)=q_{x,I}(t)=(t-\lambda)^m$.  We may assume without loss of generality that $v_0\in E_a$.
  Let $I_0$ be the
  $K$-module generated by $v_0$, and let
  $J=I/I_0$.  Since $v_0$ is a $\delta$-eigenvector, $I_0$ is
  $(x,p)$-cyclic by Lemma \ref{4.1.5}, and hence $p\mid m$. In
  particular, $q_{x,v_0}(t)=(t-\lambda)^m=(t-\lambda)^{m_0p},$ where
  $m_0\geq 1$.  Note that $I_0$ is an ideal of $L$ that is invariant
  under $\delta$.  Considering $J$ as a $K$-module, we can consider
  the Lie algebra $K \sdsum  J\cong L/I_0$ that satisfies the
  conditions of the lemma. Since $\dim J<\dim I$, the induction
  hypothesis applies to $J$, and we may write
  $J=J_1\dotplus\cdots\dotplus J_k$ where the $J_i$ are $(x,p)$-cyclic
  subspaces of $J$ and each $J_i$ is generated by a
  $\delta$-eigenvector, $w_i+I_0$, say.  Since $I_0$ has a basis
  consisting of $\delta$-eigenvectors, the $\delta$-eigenvalues in
  $J=I/I_0$ are $a,a+1,\ldots,a+p-1$.
  We claim that $w_i$ can be chosen such that $w_i\in E_a$.
  Since $x$ induces an invertible endomorphism of finite order on $I$
  (and hence on $I/I_0$),
  we have that $\left<w+I_0\right>_K=\left<x(w+I_0)\right>_K$ holds for all
  $w\in I$. Hence, by possibly swapping $w_i$ with $x^\ell w_i$ for
  some suitable $\ell\geq 0$, we may assume without loss of generality that
  $w_i+I_0$ is a $\delta$-eigenvector corresponding to eigenvalue $a$.
  As $\delta(w_i)+I_0=aw_i+I_0$,
  $\delta(w_i)-aw_i=u \in I_0 $. Since $I$ is the sum of the
  $\delta$-eigenspaces $E_{b}$,
  we may write $w_i=z_a+z_{a+1}+\cdots+z_{a+p-1}$ with $z_b \in
  E_b$.  Further, since $I_0$ is spanned by $\delta$-eigenvectors, we
  have that $I_0=(I_0 \cap E_a)\dotplus \cdots \dotplus (I_0 \cap
  E_{a+p-1})$. Thus we have $u=u_a+u_{a+1}+\cdots + u_{a+p-1}$ with
  $u_b \in E_b \cap I_0$. Hence,
  \begin{eqnarray*}\delta(w_i)-aw_i&=&az_a+(a+1)z_{a+1}+\cdots+(a+p-1)z_{a+p-1}-a(z_a+z_{a+1}+\cdots+z_{a+p-1}) \\ &=&u_a+u_{a+1}+\cdots + u_{a+p-1}.
  \end{eqnarray*}
  Since eigenvectors with different eigenvalues are linearly
  independent, we obtain $u_{a+j}= j \cdot z_{a+j}$ for all $j \geq 0$. This
  implies that $z_{a+j} = j^{-1}\cdot u_{a+j} \in I_0$ holds for all $j
  \geq 1$. Therefore $w_i= z_a+u_{a+1}+2^{-1} u_{a+2}+ \cdots
  +(p-1)^{-1}u_{a+p-1} \in z_a + I_0$. Therefore we may replace $w_i$
  by $z_a$ and so, we may assume without loss of generality that $w_i \in
  E_a$.  Since $J_i$ is $(x,p)$-cyclic,
  $q_{x,J_i}(t)=(t-\lambda)^{m_ip}$ with some $m_i\geq 1$.  We claim,
  for all $i=1,\ldots,k$, that there is some $v_i\in E_a\cap
  (w_i+I_0)$ such that
  $$ q_{x,v_i}(t)=q_{x,J_i}(t)=(t-\lambda)^{m_ip}.
  $$ We prove this claim for $i=1$.  Since
  $q_{x,J_1}(t)=(t-\lambda)^{m_1p}$, we have
  $(x-\lambda)^{m_1p}(w_1+I_0)=0$, and so $(x-\lambda)^{m_1p}(w_1)\in
  I_0$.  Thus, there is some polynomial $h\in\F[t]$ with $\deg h<m$
  and $(x-\lambda)^{m_1p}(w_1)=h(x)(v_0)$.  On the other hand, $w_1\in
  E_a$, and hence
  $$
  (x-\lambda)^{m_1p}(w_1)=[(x-\lambda)^{p}]^{m_1}(w_1)=[(x^p-\lambda^p)]^{m_1}(w_1)\in
  E_a,
  $$ which gives $h(x)(v_0)\in E_a$, since $x^p$ fixes each eigenspace
  $E_b$. Write $h(t)=h_0(t)+h_1(t)+\cdots +h_{p-1}(t)$ such
  that $$h_j(t)=a_jt^j+a_{p+j}t^{p+j}+a_{2p+j}t^{2p+j}+\cdots,$$  for all
  $ 0 \leq j \leq p-1.$ Suppose that $h_j\neq 0$ for some
  $j>0$. Observe that $h_j(x)(v_0) \in E_{a+j}$. As $h(x)(v_0) \in
  E_a$ and eigenvectors associated to different eigenvalues are
  linearly independent, we have $h_j(x)(v_0)=0$. Thus, $q_{x,v_0} \mid
  h_j$. On the other hand, $\deg h_j < m = \deg q_{x,v_0}$, which
  implies that $h_j=0$ for all $j > 0$. Hence, we can assume that
  $h=h_0=\bar{h}^p$ with some $\bar{h} \in \F[t]$. Now observe that
  $$ 0=(x-\lambda)^m(w_1)=(x-\lambda)^{m-m_1p}(x-\lambda)^{m_1p}(w_1)=
  (x-\lambda)^{m-m_1p}h(x)(v_0).
  $$ Since $q_{x,v_0}(t)=(t-\lambda)^{m}$, we have that
  $(t-\lambda)^m\mid (t-\lambda)^{m-m_1p}h(t)$, and so
  $(t-\lambda)^{m_1p}\mid h(t)$.  Therefore there is some $q \in
  \F[t]$ such that $q(t)(t-\lambda)^{m_1p}=h(t)=\bar{h}(t)^p$. This
  also implies that $q=\bar{q}^p$ with some $\bar{q}$. Now set
  $v_1=w_1-q(x)(v_0)$.  Since $q(x)(v_0)\in I_0$, we have $v_1\in
  w_1+I_0$.  Further, $q(x)(v_0)=\bar{q}(x)^p(v_0)\in E_a$, and hence
  $v_1\in E_a$.  This implies also that $q_{x,J_1}\mid q_{x,v_1}$. On
  the other hand,
  $$ (x-\lambda)^{m_1p}(v_1)=(x-\lambda)^{m_1p}(w_1-q(x)(v_0))=
  (x-\lambda)^{m_1p}(w_1)-(x-\lambda)^{m_1p}q(x)(v_0)=0.
  $$ Thus $q_{x,v_1}(t)=(t-\lambda)^{m_1p}=q_{x,J_1}(t)$, as claimed.
  For $i=1,\ldots,k$, let $I_i=\left<v_i\right>_K$. We claim that
  $I=I_0\dotplus\cdots\dotplus I_k$. First,
  $$ J_i=\left<w_i+I_0\right>_K=\left<v_i+I_0\right>_K=(I_i+I_0)/I_0
  $$ and so
  $$ I/I_0=(I_1+I_0)/I_0\dotplus \cdots\dotplus (I_k+I_0)/I_0.
  $$ This implies that $I=I_0+I_1+\cdots+I_k$. Further, the direct
  decomposition of $I/I_0$ also implies that $\dim I_0+\sum_i\dim
  (I_i+I_0)/I_0=\dim I$. On the other hand, since
  $q_{x,v_i}=q_{x,J_i}$, we also obtain that $\dim I_i=\dim J_i=\dim
  (I_i+I_0)/I_0$. Therefore
  $$ \dim I_0+\dim I_1+\cdots+\dim I_k=\dim I.
  $$ Hence the direct decomposition $I=I_0\dotplus I_1\dotplus\cdots\dotplus
  I_k$ is valid.
\end{proof}

Now we can prove Theorem~\ref{4.1.7}.

\begin{proof}[The proof of Theorem~\ref{4.1.7}]
  One direction of the theorem follows directly from
  Lemma~\ref{4.1.9}, and so we are only required to show
  the other direction.
  
  Let $\delta \in \der L$ be a non-singular derivation of finite order.
  Set $I=L'$. Note that $I$ is $\delta$-invariant and $\delta(x)=\alpha x+a$ with
  some $\alpha\in\F\setminus\{0\}$ and $a\in I$. Defining
  $\delta'\in\ennd L$ as $\delta'(a)=\delta(a)$ for $a\in I$
  and $\delta'(x)=\alpha x$ we obtain that $\delta'\in\der L$ and
  also that $I$ and $K=\left<x\right>$
  are $\delta'$-invariant. Further, if the order
  of $\delta'|_I$ is $np^r$ where $p\nmid n$, then the
  derivation $\delta''=(\delta')^{p^r}$ is non-singular and the
  order of the restriction of $\delta''$ to $I$ is coprime to $p$. Hence
  the restriction of $\delta''$ to $I$ is diagonalizable.
  In addition, multiplying $\delta''$ by a scalar we may also assume
  that $\delta''(x)=x$. 

  By the argument in the previous paragraph, we may assume without loss
  of generality that
 $L$ has a non-singular derivation $\delta$ of finite order such that
$\delta( x )=x $, $\delta(I)=I$ and that
the restriction of $\delta$ to $I$ is diagonalizable.    Let $q_{x,I}(t)=(t-\lambda_1)^{m_1}\cdots
(t-\lambda_k)^{m_k}$ be the minimal polynomial of $x$ considered as an element of
$\gl I$. As $K$ is one-dimensional, the collected primary
decomposition of $I$ into $K$-modules is
$I=I_0(x-\lambda_1)\dotplus\cdots\dotplus
I_0(x-\lambda_k)$, and Lemma~\ref{5.3} implies that the
$I_0(x-\lambda_i)$ are $\delta$-invariant.  Hence we may assume
without loss of generality  that $k=1$ and
$q_{x,I}(t)=(t-\lambda)^m$. Further, $I$ can be decomposed as $I=\bar
E_{a_1}\dotplus\cdots\dotplus \bar E_{a_s}$ where, for $a_i\in\F$,
$\bar E_{a_i}$ is the sum of the eigenspaces
$E_{a_i},\ldots,E_{a_i+p-1}$ of $\delta$ that
correspond to the eigenvalues $a_i,a_i+1,\ldots,a_i+p-1$, respectively.
Since $x$ is a $\delta$-eigenvector corresponding to eigenvalue $1$,
$x( E_{a_i+j})\leq  E_{a_i+j+1}$ holds for all $j\in\{0,\ldots,p-2\}$
and $x( E_{a_i+p-1})\leq  E_{a_i}$. Thus each $\bar E_{a_i}$ is $x$-invariant.
Therefore we may
assume that $I=\bar E_a$ with some $a\in \F$. Now the theorem follows
from Lemma~\ref{4.1.19}.
\end{proof}

We present an application of Theorem~\ref{4.1.7} concerning
the representations of the
Heisenberg Lie algebra with the property that the semidirect sum admits
a non-singular derivation.

\begin{teo}\label{7.3} Let $\F$ be the algebraic closure of
  $\F_p$ where $p$ is an odd prime. Let $H$ be the Heisenberg Lie algebra over
  $\F$. Let $\psi:H \to \gl I$ be a faithful representation,
  and suppose that
  $L=H \sdsum I$ is non-nilpotent.  If $L$ has a non-singular derivation
   of finite order such that $\delta(I)\leq I$, then $\dim I \geq p+3$.
\end{teo}

 \begin{proof} Let $\delta \in \der L$ be a non-singular
   derivation of finite order such that $\delta(I)\leq I$.
   Using the argument at the beginning of the proof of Theorem~\ref{4.1.7},
   we can suppose that $\delta$ is diagonalizable, $\delta(I)=I$
   and $\delta(H)=H$. Hence $L$ can be viewed as a graded Lie algebra
   over the additive group of $\F$ as explained after
   Lemma~\ref{NonSDiag}.
 As $L$ is non-nilpotent, by Lemma \ref{4.2.4}, there are
 homogeneous $\delta$-eigenvectors $k,\ v \in L$ such that $[k^{m}, v] \neq 0$
 for all $m>0$. Write $k=k_H+k_I$ and $v=v_H + v_I$, such that
 $k_H,\ v_H \in H$ and $k_I,\ v_I \in I$. Notice that $k_H,\ v_H,\ k_I,\ v_I$
 are $\delta$-eigenvectors. In fact, if $\delta(k)=bk$, for some
 $b \in \F$,
 then $$\delta(k_H)+\delta(k_I)=\delta(k_H+k_I)=b(k_H+k_I).$$ Hence,
 as $H$ and $I$ are invariant under $\delta$ and $L= H \sdsum I$,
 $\delta(k_H)=bk_H$ and $\delta(k_I)=bk_I$. Analogously,  if
 $\delta(v)=cv$ for some $c \in \F$, then $\delta(v_H)=cv_H$ and
 $\delta(v_I)=cv_I$.  We claim that there is $h \in H$ and $v \in I$
 such that $[h^m,v] \neq 0$ for all $m \geq 1$. We have
 that,  $$[k^{m}, v_H+v_I] \neq 0$$ for all $m \geq 1$, and so,
 $[k^{m}, v_H] \neq 0$ or $[k^{m}, v_I] \neq 0$ for all $m \geq 1$.
 Suppose first that $[k^{m}, v_H] \neq 0$ for all $m \geq 1$.
 Then, since $[k_H,v_H] \in Z(H)$ and since $I$ is abelian,
\begin{eqnarray*} [k^{2},v_H]  &=&  [k_H+k_I,[k_H+k_I,v_H]]
     =  [k_H+k_I,[k_H,v_H]+[k_I,v_H]]      \\&=&
      [k_H,[k_H,v_H]]+[k_H,[k_I,v_H]]+[k_I,[k_H,v_H]]+[k_I,[k_I,v_H]]
      \\   &=&  [k_H,[k_I,v_H]]+[k_I,[k_H,v_H]].
 \end{eqnarray*}
This means that $[k^2,v_H] \in
  I$. Therefore, as $I$ is
  abelian, $$[k^3,v_H]=[k,[k^2,v_H]]=[k_H+k_I,[k_2,v_H]]=[k_H,[k_2,v_H]]
$$ and easy induction shows that 
$$[(k_H)^m,[k^2,v_H]]=[k^{m+2},v_H] \neq 0$$ holds for all $m \geq
  1$. Hence the choice of $h=k_H$ and $v=[k^2,v_H]$ is as claimed.
  If $[k^m,v_I] \neq 0,$ for all $m \geq 1$, then let $v=v_I$ and
  let $h=k_H$. Hence, an argument similar to the one in the previous
  case shows that $[h^m,v] \neq 0$ for all $m \geq 1$.  

Let $h \in H$ and $v \in I$
 be $\delta$-eigenvectors such that  $[h^m,v] \neq 0$, for all $m \geq
1$, as  in the previous paragraph.  Let $ q $ be the minimal polynomial of $\psi(h)$
as element of $\ennd I$ and suppose that $ q = q_1^{s_1} \ldots
q_r^{s_r} $ is the factorization of $q$ into irreducible
factors. Then, by the argument presented at the beginning of Section~\ref{prim.dec.sec}, $I$ can be written as
the direct sum $ I = I_0 (q_1 (h)) \dotplus \ldots \dotplus I_0 (q_r
(h)) $.  By Lemma~\ref{linear.05}, each $I_0 (q_i (h))$ is an
$H$-module. Let $ I_1$ be the sum of $I_0 (q_i (h))$ such that $
q_i(t) \neq t$, and set $ I_0 = I_0 (h)$. Thus, $$L=H \sdsum \left(
I_0 \dotplus I_1 \right).$$ Also, $I_1 \neq 0,$ since $v \not\in I_0$. By
Lemma~\ref{5.3}, $ I_0 $ and $ I_1 $ are $H$-modules and
$\delta$-invariant. It follows that the Lie algebras $L_0=H \sdsum
I_0$ and $L_1=H \sdsum I_1$ have a non-singular derivation. Observe
that, by the construction of $L_1$, $h$ acts non-singularly on
$I_1$. Hence, $L_1$ is non-nilpotent.  Let $\delta_1$ be the
restriction of $\delta$ to $L_1$. The derivation $\delta_1 \in \der
{L_1}$ is non-singular and has finite order. As $h$ is an eigenvector
of $\delta_1$ and $I$ is $\delta$-invariant, the Lie algebra $\langle
h \rangle \sdsum I$ is $\delta$-invariant, and so the restriction of
$\delta$ to $\langle h \rangle \sdsum I$ is a non-singular derivation
of finite order.  Note that
$h$ induces a non-singular linear transformation on $I_1$
and this transformation has  finite order, since $\F$ is the algebraic
closure of $\F_p$, and so all the entries of the matrix representing $h$ in
a basis of $I_1$ are elements of a suitable finite field.
Thus, by
Theorem \ref{4.1.7}, $I_1$ can be written as a direct sum of
$(h,p)$-cyclic modules, and so $\dim I_1=np$ for some  $n \geq 1$.
In particular, $\dim I_1\geq p$.

The
action of $H$ must be faithful either on $I_1$ or on $I_0$. For, if
$H$ were not faithful on $I_0$ and on $I_1$, then $Z(H)$ would act
trivially on both $I_1$ and $I_0$, and hence $Z(H)$ would act trivially on
$I$. This contradicts to the assumption that $I$ is a faithful
$H$-module.  As $\delta(H)=H$, $\delta(Z(H))=Z(H)$. Hence, if $z \in
Z(H)\setminus \{0\}$, then $\delta (z)=d z,$ since $\dim(Z(H))=1$.  If
$I_1$ is a faithful module, then there is $u \in I_1$ a
$\delta$-eigenvector associated to the eigenvalue $e \in \F$, such
that $[z,u] \neq0$. It follows that $\delta([z,u])=(d+e)[z,u]$. Then,
since $d \neq 0$, $u$ and $[z,u]$ are linearity independent. If $\dim
I_1 = p$, then by~\cite[Corollary~2.4]{Szechtman} the representation
is irreducible and there exists $f \in \F$ such that $[z,w]=f w$ for
all $w \in I$. This, however, contradicts the fact that $u$ and $[z,u]$ are
linearity independent, and so $\dim I_1 \geq 2p$. As $p \geq 3$, $\dim
{I_1} \geq p+3$. If $ I_1 $ is not faithful, then $ I_0 $ is, and, by
\cite[Theorem~3.1]{Szechtman}, $ \dim {I_0} \geq 3$. In both
cases, $\dim I= \dim {I_0}+\dim {I_1} \geq p+3$.
\end{proof}

 The following example shows that the bound $p+3$ in Theorem~\ref{7.3}
 is sharp.

 \begin{ex}
   Let $\F$ be a field of characteristic $p$, and let $I$ be a vector
   space of dimension $p+3$ over $\F$ with basis $B=\{v_1,v_2,v_3, u_0,u_1, \ldots, u_{p-1}\}$.
   We define a representation of the Heisenberg Lie algebra $H=\left<x,y,z
   \right>
   $ on $I$ as follows:
   \begin{eqnarray*}
x&:& v_2 \mapsto v_1,  u_0 \mapsto u_1, u_1 \mapsto u_2, \ldots, u_{p-1} \mapsto u_0;\\
y&:&v_3 \mapsto v_2;\\
z&:&v_3 \mapsto v_1.\\
\end{eqnarray*}
   The Lie algebra $L=H \sdsum I$ is not nilpotent, but solvable
   with derived length~3. Suppose that $a,\ b,\ c,\ d\in \F$ and 
   define $\delta: L \to L$ by $\delta(x)=ax$,
   $\delta(y)=by$, $\delta(z)=(a+b)z$, $\delta(u_i)=(c+ia)$,
   $\delta(v_1)=(a+b+d)v_1$, $\delta(v_2)=(b+d)v_2$ and $\delta(v_3)=dv_3$.
   If $a,\ b,\ c,\ d$ are chosen so that the eigenvalues of $\delta$ are
   non-zero, then  $\delta$ is
   a non-singular derivation of $H\sdsum I$. The details can be verified as
   in Proposition~\ref{6.2.3} and are left to the reader.
\end{ex}

\end{document}